\documentclass{amsart}

\usepackage{amssymb}
\usepackage[abbrev]{amsrefs}
\usepackage{hyperref}
\usepackage{enumitem}

\newtheorem{theorem}{Theorem}[section]
\newtheorem{conjecture}[theorem]{Conjecture}
\newtheorem{proposition}[theorem]{Proposition}
\newtheorem{lemma}[theorem]{Lemma}

\theoremstyle{definition}
\newtheorem{definition}[theorem]{Definition}

\theoremstyle{remark}
\newtheorem{remark}[theorem]{Remark}

\newcommand{\F}{\mathbb F}
\newcommand{\Z}{\mathbb Z}
\newcommand{\f}{\mathfrak f}

\DeclareMathOperator{\cha}{char}
\DeclareMathOperator{\ord}{ord}
\DeclareMathOperator{\Gal}{Gal}
\DeclareMathOperator{\Tr}{Tr}
\newcommand{\C}{\mathcal C}

\usepackage{longtable}
\usepackage{etoolbox}
\makeatletter
\patchcmd{\LT@makecaption}%
   {\@tempboxa{#1{#2: }#3}}
   {\@tempboxa{#1{\textsc{#2. }}#3}}{}{}
\patchcmd{\LT@makecaption}%
   {#1{#2: }#3}
   {#1\textsc{#2. }{#3}}{}{}
\makeatother

\begin{document}
  \title[Translates of completely normal elements]{Translates of completely normal elements and the Morgan-Mullen conjecture}
  
  \author{Theodoulos Garefalakis}
  \address{Department of Mathematics \& Applied Mathematics, University of Crete, Voutes Campus, 70013 Heraklion, Greece}
  \email{tgaref@uoc.gr}
  \author{Giorgos Kapetanakis}
  \address{Department of Mathematics, University of Thessaly, 3rd km Old National Road Lamia-Athens, 35100 Lamia, Greece}
  \email{kapetanakis@uth.gr}
  \date{\today}
  \begin{abstract}
    Denote by $\F_q$ the finite field of order $q$ and by $\F_{q^n}$ its extension of degree $n$. Some $a\in\F_{q^n}$ is called \emph{primitive} if it generates the multiplicative group $\F_{q^n}^*$ and it is called \emph{$q^n/q$-normal} if its $\F_q$-conjugates form an $\F_q$-basis of $\F_{q^n}$ if the latter is viewed as an $\F_q$-vector space. Furthermore, some $a\in\F_{q^n}$ is called \emph{$q^n/q$-completely normal} if it is $q^n/q^d$-normal for all $d\mid n$. In this work we prove a new construction of sets of completely normal elements and we establish, under conditions, the existence of elements that are simultaneously primitive and $q^n/q$-completely normal, covering some yet unresolved cases of a 30-year-old conjecture by Morgan and Mullen.
  \end{abstract}
  \maketitle
  \section{Introduction} \label{sec:intro}
  Throughout this text $q$ will denote a power of the prime $p$ and $n$ a positive integer. Further, $\F_q$ will stand for the finite field of order $q$ and $\F_{q^n}$ for its extension of degree $n$, while for any $d\mid n$, $\F_{q^d}$ will stand for the corresponding intermediate extension.
  
  Some $a\in\F_{q^n}$ is called \emph{primitive} if it generates the multiplicative group $\F_{q^n}^*$. It is well-known that every finite field contains primitive elements, see \cite{lidlniederreiter97}*{Theorem~2.8}, while primitive elements of finite fields are important for both theoretical and practical reasons, with their most notable applications found in areas such as cryptography and coding theory.
  
  On the other hand, some $a\in\F_{q^n}$ is called \emph{$q^n/q$-normal} (or \emph{normal over $\F_q$} or just \emph{normal}) if its $\F_q$-conjugates, that is the set 
  \[ \{\sigma(a) : \sigma\in\Gal(\F_{q^n}/\F_q) \} = \{a, a^q,\ldots ,a^{q^{n-1}}\}, \]
   forms an $\F_q$-basis of $\F_{q^n}$ when the latter is viewed as an $\F_q$-vector space. A basis of this form is called an \emph{$q^n/q$-normal basis}. The existence of these elements is known.
  \begin{theorem}[Normal Basis Theorem] \label{thm:nbt}
     There exists some $a\in\F_{q^n}$ that is $q^n/q$-normal.
  \end{theorem}
  For a proof of the above, we refer the reader to \cite{lidlniederreiter97}*{Theorem~2.35}. Although primitive elements are more commonly employed for practical purposes than normal elements, the latter category is important for both theoretical and practical reasons on its own right, with their most notable applications found in areas where efficient finite field arithmetic is important, such as coding theory or radar technology.
  
  A natural question is whether there exist elements of $\F_{q^n}$ that combine both of the above properties. The answer to that question was given by the following celebrated result.
  \begin{theorem}[Primitive Normal Basis Theorem] \label{thm:pnbt}
    There exists some $a\in\F_{q^n}$ that is simultaneously primitive and $q^n/q$-normal.
  \end{theorem}
  The above was initially established in 1987 by Lenstra~Jr{.} and Schoof~\cite{lenstraschoof87}, but Cohen and Huczynska~\cite{cohenhuczynska03} gave a computer-free proof in 2003. Both of these works employ the \emph{character sum method}, which we will also employ in this work. Roughly speaking, this method is used in existence questions. In its typical application, the number of elements of interest is expressed as a sum that involves characters. Then, the sum is split in two terms: the \emph{main term} and the \emph{error term}. The main term corresponds to all the characters being trivial and the error term is the remaining sum. The existence assertion follows, should the difference of the two terms be nonzero.
  
  A natural question that arises from Theorem~\ref{thm:nbt} is whether some $a\in\F_{q^n}$ can be normal over the intermediate extensions of $\F_{q^n}/\F_q$. In particular, $a\in\F_{q^n}$ is called \emph{$q^n/q$-completely normal} (or just \emph{completely normal}) if it is $q^n/q^d$-normal for all $d\mid n$ and the existence of such elements has also been established.
  \begin{theorem}[Completely Normal Basis Theorem] \label{thm:cnbt}
     There exists some $a\in\F_{q^n}$ that is $q^n/q$-completely normal.
  \end{theorem}
  The above was initially established in 1986 by Blessenohl and Johnsen~\cite{blessenohljohnsen86}, but with general Galois extensions in mind. Later on, in 1994, Hachenberger~\cite{hachenberger94} gave a simplified proof that is specific to finite field extensions.
  
  A simple glimpse to Theorems~\ref{thm:nbt}, \ref{thm:pnbt} and \ref{thm:cnbt} immediately leads to wondering about the existence of elements of $\F_{q^n}$ that are, at the same time, primitive and completely normal. Thus, we say that the extension $\F_{q^n}/\F_q$ possesses the \emph{MM property} if it contains an element $a\in\F_{q^n}$ that is primitive and completely normal. In fact, in 1996 by Morgan and Mullen~\cite{morganmullen96} stated the following.
  \begin{conjecture}[Morgan-Mullen] \label{conj:mm}
     Every finite field extension possesses the MM property.
  \end{conjecture}
  The above, as we will see in some detail in Section~\ref{sec:previous}, has been partially established. 
  In particular, we know that certain families possess the MM property and the conjecture has been directly verified by computational means, for numerous \emph{small} finite field extensions, while no counterexample has been found. This evidence strongly suggests the validity of this conjecture, which is, however, not yet fully resolved.
  
  The aim of this work is to make an additional step towards the proof of Conjecture~\ref{conj:mm}, by establishing that new families of finite field possess the MM property, by adding a new weapon to the arsenal found in the literature, for the researcher willing to attack this problem, and by demonstrating that new algebraic constructions of completely normal elements can be used to attack Conjecture~\ref{conj:mm}.
   
  Towards this end,
  we introduce the family of partially completely basic finite fields extensions and establish that, under conditions, these extensions possess the MM property.
  
  This paper is organized as follows.
  In Sect.~\ref{sec:previous} we present the known conditions under which Conjecture~\ref{conj:mm} is known to hold.
  In Sect.~\ref{sec:prelim} we present some background material that we will use along the way.
  In Sect.~\ref{sec:construction} we establish a novel construction of sets of completely normal elements, see Proposition~\ref{propo:a1a2-1} and Theorem~\ref{thm:completely_normal_translates}.
  In Sect.~\ref{sec:main} we introduce the family of \emph{$(n_1,n_2)$-partially completely basic} extensions of finite fields and prove novel conditions that ensures the existence of elements that are primitive and $q^n/q$-completely normal for such extensions, see Theorems~\ref{thm:main_condition} and \ref{thm:main_condition_2}.
  In Sect.~\ref{sec:asymptotic} we establish an asymptotic result for $(n_1,n_2)$-partially completely basic extensions to possess the MM property (see Theorem~\ref{thm:main_asymptotic}).
  In Sect.~\ref{sec:comp} we explore the computational aspects of our work. In particular, we describe an algorithm that takes advantage of our theoretical arguments to establish the MM property for $(n_1,n_2)$-partially completely basic extensions, for fixed $n_1$, see Subsect.~\ref{subsec:algo}. Then, in Subsect.~\ref{subsec:concrete} we use the algorithm to settle the $n_1=2$ and $n_1=3$ cases, see Theorem~\ref{thm:main_concrete}, and comment on the $n_1=4$ case.
  We conclude this work with a small discussion about future additions or improvements to this line of research in Sect.~\ref{sec:discussion}.
  \section{Towards the Primitive Completely Normal Basis Theorem} \label{sec:previous}
We will split our exhibition of previous results depending on the means that workers that attacked Conjecture~\ref{conj:mm} have adopted in the past, rather than a strictly chronological order.
%
%
The first and most obvious method is direct verification with the help of computers. This is in fact the method that Morgan and Mullen~\cite{morganmullen96} used in order to support Conjecture~\ref{conj:mm}.
\begin{theorem}[Morgan-Mullen] \label{thm:mm}
  Suppose that $q\leq 97$ and $q^n<10^{50}$. Then $\F_{q^n}/\F_q$ possesses the MM property.
\end{theorem}
Later still, with the help of more modern equipment, advanced computational methods and theoretical arguments, Hachenberger and Hackenberg~\cite{hachenbergerhackenberg19} extended the original Morgan-Mullen verifications as follows.
\begin{theorem}[Hachenberger-Hackenberg] \label{thm:hh}
  Suppose that
  \begin{enumerate}
    \item $n\leq 202$ or
    \item $q< 10^4$ and $q^n < 10^{80}$.
  \end{enumerate}
  Then, $\F_{q^n}/\F_q$ possesses the MM property.
\end{theorem}
%
%
However, most of the advances towards establishing Conjecture~\ref{conj:mm} have been achieved with theoretical means. A first observation is that all completely normal elements are, by definition, normal, but the inverse is true only for certain finite field extensions. So, the following definition is essential.
\begin{definition} \label{def:comp_basic_extensions}
Let $\F_{q^n}/\F_q$ be such that every $q^n/q$-normal element is $q^n/q$-completely normal. Then the extension $\F_{q^n}/\F_q$ is called \emph{completely basic}.
\end{definition}
For the case of completely basic extensions, Conjecture~\ref{conj:mm} is directly implied by Theorem~\ref{thm:pnbt}, thus, if $\F_{q^n}/\F_q$ is completely basic, it possesses the MM property. Clearly, if $r$ is prime, $\F_{q^r}/\F_q$ is completely basic, but this is not the only family of completely basic extensions. In fact, we have the following complete characterization of this family by Blessenohl and Johnsen~\cite{blessenohljohnsen91}.
\begin{theorem}[Blessenohl-Johnsen] \label{thm:completely_basic_characterization}
The finite field extension $\F_{q^n}/\F_q$ is completely basic if and only for every prime divisor $r$ of $n$, $r\nmid d$, where $d$ stands for the order of $q$ modulo the relatively prime to $q$ part of $n/r$.
\end{theorem}
In particular, see \cite{hachenberger13}*{Example~5.4.19}, this family includes cases such as $n=r^2$ (where $r$ is a prime number), $n\mid q-1$ or $n=p^k$, where $k$ is arbitrary and $p=\cha{\F_q}$.

Another family that is of interest for us is the following.
\begin{definition}
If $\gcd(n,d)=1$, where $d$ is the order of $q$ modulo the product of the prime divisors of $n$ that do not divide $q$, then $\F_{q^n}/\F_q$ is called a \emph{regular extension}.
\end{definition}
Completely basic extensions are clearly regular extensions. But, see \cite{hachenberger13}*{Example~5.4.47}, this family also include extensions $\F_{q^n}/\F_q$, where the set $D = \{ r\mid n : r\text{ prime} \}$ satisfies
\begin{enumerate}
  \item $|D|=1$, or, 
  \item for every $r\in D$ we have that $r\mid q-1$ or $r\mid q$, or,
  \item $D\subseteq \{7,$ $ 11,$ $ 13,$ $ 17,$ $ 19,$ $ 31,$ $ 41,$ $ 47,$ $ 49,$ $ 61,$ $ 73,$ $ 97,$ $ 101,$ $ 107,$ $ 109,$ $ 139,$ $ 151,$ $ 163,$ $ 167,$ $ 173,$ $ 179,$ $ 181,$ $ 193\}$.
\end{enumerate}
Also, $\F_{q^n}/\F_q$ is regular if $n$ is a power of a Carmichael number.
It has been established that regular extensions possess the MM property.
\begin{theorem}[Hachenberger] \label{thm:regular_extensions}
Regular extensions possess the MM property.
\end{theorem}
The above was initially partially established by Hachenberger~\cites{hachenberger01,hachenberger10} and, subsequently, fully established by the same author \cite{hachenberger19}. Furthermore, we note that, the establishment of Theorem~\ref{thm:regular_extensions} required deep understanding of the underlying module structure of finite fields and combines algebraic and character sum methods.

Last but not least, some efforts have been made with the adoption of combinatorial and character sum methods. Namely, in 2016, Hachenberger~\cite{hachenberger16} used elementary combinatorial arguments in order to obtain the following.
\begin{theorem}[Hachenberger] \label{thm:hb1}
Suppose that
\begin{enumerate}
  \item $q\geq n^{7/2}$ and $n\geq 7$, or,
  \item $q\geq n^3$ and $n\geq 37$,
\end{enumerate}
then $\F_{q^n}/\F_q$ possesses the MM property.
\end{theorem}
It is worth mentioning that the results of \cite{hachenberger16} depend on an estimate of the number of $q^n/q$-completely normal elements, that derives from combinatorial arguments and does not depend on the prime decomposition of $n$ and its relation with $q$.

Inspired by this work, the authors adjusted the character sum method accordingly for this setting. More precisely, in \cites{garefalakiskapetanakis18b,garefalakiskapetanakis18}, the main term of the authors' application of the method was manipulated in a way that it equaled the number of $q^n/q$-completely normal elements. This lead to the following improvement of Theorem~\ref{thm:hb1} in \cite{garefalakiskapetanakis18}.
\begin{theorem}[Garefalakis-Kapetanakis] \label{thm:gk1}
  If the part of $n$ that is relatively prime to $q$ is smaller than $q$, then $\F_{q^n}/\F_q$ possesses the MM property. In particular, all extensions $\F_{q^n}/\F_q$ with $n\leq q$ possess the MM property.
\end{theorem}
Soon after, with more attention given to technical details, the above was further improved in \cite{garefalakiskapetanakis18b} by the same authors.
\begin{theorem}[Garefalakis-Kapetanakis] \label{thm:gk2}
 If either
 \begin{enumerate}
   \item $n$ is odd and $n<q^{4/3}$, or,
   \item $n$ is even, $q-1\nmid n$ and $n<q^{5/4}$,
 \end{enumerate}
 then $\F_{q^n}/\F_q$ possesses the MM property.
\end{theorem}
Here, we remark that, in this section, emphasis was given to the end results regarding the existence of primitive and $q^n/q$-completely normal elements, which is also the object of this work. However, in addition to this question, there are many interesting related questions, such as the distribution of these elements. We refer the interested reader to the references of this work and in particular to related textbooks \cites{hachenberger97,hachenberger13,hackenbergerjungnickel20} and the references therein.
  \section{Preliminaries} \label{sec:prelim}
As we mentioned in Sect.~\ref{sec:intro}, in this work we will take advantage of the character sum method. In this method, characters and their sums play a crucial role in characterizing elements of finite fields with the desired properties and in estimating the number of elements that combine all the desired properties.
\begin{definition} \label{def:char}
Let $\mathfrak{G}$ be a finite abelian group. A \emph{character} of
$\mathfrak{G}$ is a group homomorphism $\mathfrak{G} \to \mathbb{C}^*$. The characters of $\mathfrak{G}$ form a group under multiplication,
which is isomorphic to $\mathfrak{G}$. This group is called the \emph{dual} of
$\mathfrak{G}$ and denoted by $\widehat{\mathfrak{G}}$. Furthermore, the
character $\chi_0 : \mathfrak{G} \to \mathbb{C}^*$, where $\chi_0(g) = 1$ for
all $g\in \mathfrak{G}$, is called the \emph{trivial
character} of $\mathfrak{G}$. Finally, by $\bar\chi$ we denote the inverse of $\chi$ and observe that, for every $g\in\mathfrak G$ and $\chi\in\widehat{\mathfrak G}$, $|\chi(g)|=1$.
\end{definition}
The finite field $\F_{q^n}$ is associated with its multiplicative and its additive group. From now on, we will call the characters of $\F_{q^n}^*$ \emph{multiplicative
characters} and the characters of $\F_{q^n}$ \emph{additive characters}.
Furthermore, we will denote by $\chi_0$ and $\psi_0$ the trivial multiplicative
and additive character respectively and we will extend the multiplicative
characters to zero with the rule
\[ \chi(0) := \begin{cases} 0, & \text{if }\chi\in\widehat{\F_{q^n}^*}
\setminus \{ \chi_0 \} , \\
1, & \text{if } \chi = \chi_0 . \end{cases} \]
A \emph{character sum} is a sum that involves characters. In this work we will use the following incomplete character sum estimate.
%
\begin{theorem} \label{thm:char-sum}
Let $\F_{q^n}/\F_q$ be a finite field extension and take $\chi$ a multiplicative character of $\F_{q^n}$ and $\psi$ an additive character of $\F_q$, such that not both of them are trivial. Further, let $a\in\F_{q^n}$ be such that $\F_q(a)=\F_{q^n}$. Then
\[ \left| \sum_{c\in\F_q} \chi(a +c)\psi(c) \right| \leq n q^{1/2} . \]
\end{theorem}
The above follows from the main result of \cite{fuwan14}. However, an alternative simplified proof was established recently, see \cite{mazumderkapetanakiskalabasnet25}.

The additive and the multiplicative groups of $\F_{q^n}$ can also be seen as modules. In particular $\F_{q^n}^*$ (the multiplicative group) can be seen as a $\Z$-module under the rule $r\circ a = a^r$, where $r\in\Z$ and $a\in\F_{q^n}$. The fact that $\F_{q^n}$ contains primitive elements implies that this module is cyclic.


By using Vinogradov's formula for generators of cyclic modules over Euclidean domains, it follows that the characteristic function for $a\in\F_{q^n}$ being primitive is
\[
\omega(a) := \theta(q') \sum_{d\mid q'} \frac{\mu(d)}{\phi(d)} \sum_{\chi\in\widehat{\F_{q^n}^*},\ \ord(\chi)=d} \chi(a),
\]
where $q'=q^n-1$, $\theta(r)=\phi(r)/r$, $\mu$ is the M\"obius function, $\phi$ is the Euler function and the \emph{order} of a multiplicative character is defined as its multiplicative order in $\widehat{\F_{q^n}^*}$. Also, we stress that $\F_{q^n}^*$ contains $\phi(q')$ primitive elements.

Regarding the additive group, for any $d\mid n$, $\F_{q^{d}}$ can be seen as an $\F_{q}[x]$-module under the rule $f\circ a = \f(a)$, where $\f$ stands for the \emph{$q$-associate} of $f\in\F_{q}[x]$, that is, for $f=\sum_{i=0}^k f_ix^i$, then
 $\f = \sum_{i=0}^k f_ix^{q^{i}}$.
 
Theorem~\ref{thm:nbt} implies that for all $d\mid n$, the corresponding module is cyclic.  Similarly, the characteristic function for $a\in\F_{q^{d}}$ being ${q^{d}}/{q}$-normal is
\[
\varOmega_{d}(a) := \theta_q(F_d') \sum_{f\mid F_{d}'} \frac{\mu_q(f)}{\phi_q(f)} \sum_{\psi\in\widehat{\F_{q^d}}, \ \ord(\psi)=f} \psi(a) ,
\]
 where $F_{d}'=x^{d}-1\in\F_{q}[x]$, $\theta_{q}(f):= \phi_{q}(f)/q^{\deg(f)}$, $\mu_{q}$ and $\phi_{q}$ are the M\"obius and Euler functions in $\F_{q}[x]$, respectively, the first sum extends over the monic divisors of $F_d'$ in $\F_{q}[x]$ and the second sum runs through the additive characters of $\F_{q}$ of $q$-order $f$. The \emph{$q$-order} of some $\psi\in\widehat{\F_{q^{d}}}$, denoted by $\ord(\psi)$, is defined as the lowest degree monic polynomial $g\in\F_{q}[x]$ such that $\psi( g\circ a ) = 1$ for all $a\in\F_{q^{d}}$. We note that the $q$-order of an additive character of $\F_{q^{d}}$ divides $F_d'$ and that $\F_{q^d}$ contains $\phi_q(F_d')$ elements that are $q^d/q$-normal.

%

Last but not least, it follows immediately from the orthogonality relations that the characteristic function for some $a\in\F_q$ to be nonzero is
\[ z(a) = \frac 1q \sum_{\eta\in\widehat{\F_q}}(1- \eta(a)) . \]
\section{A construction of completely normal elements} \label{sec:construction}
Before we proceed, we start with the following definition.
\begin{definition}
Take some finite field extension $\F_{q^n}/\F_q$ and some $a\in\F_{q^n}$. The \emph{set of translates} of $a$ is defined as 
\[ \mathcal T_a = \{ a + c : c\in\F_q\} . \]
\end{definition}
In this section we will modify the following proposition, that plays an important role in the original proof of Theorem~\ref{thm:cnbt}.
\begin{proposition}[\cite{hachenberger94}*{Claim~(3.1.4)}] \label{propo:hb2}
Assume that $n=n_1n_2$ where $n_1, n_2$ are relatively prime. If $a_i\in\F_{q^{n_i}}$ is $q^{n_i}/q$-completely normal for $i=1,2$, then $a_1a_2$ is  $q^n/q$-completely normal.
\end{proposition}
Towards this end, we will need the following auxiliary lemmas and the notion of the trace function. The \emph{$q^n/q$-trace} is the function
\[ \Tr_{q^n/q} : \F_{q^n}\to\F_q , \ a \mapsto \sum_{i=0}^{n-1} a^{q^i} = \sum_{\sigma\in\Gal(\F_{q^n}/\F_q)} \sigma(a) . \]
\begin{lemma}[\cite{hachenberger94}*{Claim~(3.1.2)}] \label{lemma:hb1}
Let $n_1,n_2$ be relatively prime and assume that some $a\in\F_{q^{n_1}}$ is $q^{n_1}/q$-normal. Then $a$ is $q^{n_1n_2}/q^{n_2}$-normal (as an element of $\F_{q^{n_1n_2}}$).
\end{lemma}
\begin{lemma}[\cite{hachenberger94}*{Claim~(3.1.3)}] \label{lemma:hb2}
Assume that $n=n_1n_2$ where $n_1, n_2$ are relatively prime. If $a_i\in\F_{q^{n_i}}$ is $\F_{q^{n_i}}/\F_q$-normal for $i=1,2$, then $a_1a_2$ is  $q^n/q$-normal.
\end{lemma}
\begin{lemma}[\cite{mazumderkapetanakisbasnet25}*{Lemma~3.1}] \label{lemma:normal_trace}
Suppose $a\in\F_{q^n}$ is $q^n/q$-normal and $d\mid n$. Then $\Tr_{q^n/q^d}(a)$ is $q^d/q$-normal. In particular, $\Tr_{q^n/q^d}(a)\neq 0$.
\end{lemma}
\begin{lemma} \label{lemma:normal_mapping}
Suppose $a\in\F_{q^n}$ is $q^n/q$-normal. Choose some $\lambda_i\in\F_q$ for $i=0,\ldots ,n-1$. Then $\sum_{i=0}^{n-1} \lambda_i a^{q^i} \in\F_q$ if and only if all of the $\lambda_i$'s are equal.
\end{lemma}
\begin{proof}
Since $a$ is $q^n/q$-normal, the map
\[ \F_q^n \to \F_{q^n} ,\, (\lambda_0,\ldots ,\lambda_{n-1}) \mapsto \sum_{i=0}^{n-1} \lambda_i a^{q^i} \]
is a bijection. Thus, it suffices to prove the inverse claim. For the inverse, observe that, if $\lambda_i=\lambda$ for all $i=0,\ldots ,n-1$, then
\[ \sum_{i=0}^{n-1} \lambda_i a^{q^i} = \lambda \Tr_{q^n/q}(a)\in\F_q , \]
while Lemma~\ref{lemma:normal_trace} yields that $\Tr_{q^n/q}(a)\neq 0$.
\end{proof}
%
%

We prove the following lemma.
\begin{lemma} \label{lemma:x+1}
Suppose $a\in\F_{q^n}$ is $q^n/q$-normal and take some $c\in\F_q$. If $p\mid n$, then $a+c$ is $q^n/q$-normal. If $p\nmid n$, then the following are equivalent:
\begin{enumerate}
\item $a+c$ is $q^n/q$-normal.
\item $\Tr_{q^n/q}(a+c)\neq 0$.
\item $c \neq -\Tr_{q^n/q}(a)\cdot n^{-1}$, where $n$ is seen as an element of $\F_q$.
\end{enumerate}
\end{lemma}
\begin{proof}
Take
\[ \sum_{i=0}^{n-1} \lambda_i (a+c)^{q^i} = 0 , \]
for some $\lambda_i\in\F_q$, for $i=0,\ldots ,n-1$. Recall that $a+c$ is $q^n/q$-normal if and only if the above holds only for $\lambda_i=0$ for all $i=0,\ldots ,n-1$. Now, the latter is equivalent to
\[ \sum_{i=0}^{n-1} \lambda_i a^{q^i}  =- c \sum_{i=0}^{n-1} \lambda_i . \]
Further, Lemma~\ref{lemma:normal_mapping} entails that $\lambda_i=\lambda$ for every $i=0,\ldots ,n-1$, for some $\lambda \in\F_q$. Finally, the last equation becomes
\[ \lambda \left( \Tr_{q^n/q}(a) +cn\right) = 0 , \]
which implies the desired result, with Lemma~\ref{lemma:normal_trace} in mind.
\end{proof}
%
%
\begin{proposition} \label{propo:a1a2-1}
Assume that $n=n_1n_2$ where $n_1$ and $n_2$ are relatively prime. If $a_i\in\F_{q^{n_i}}$ is $q^{n_i}/q$-completely normal
for $i=1,2$.
If $p\mid n$, then $a_1a_2+c$ is  $q^n/q$-completely normal for every $c\in\F_q$. If $p\nmid n$, then, for every $c\in\F_q$, the following are equivalent:
\begin{enumerate}
\item $a_1a_2+c$ is $q^n/q$-completely normal.
\item $\Tr_{q^n/q}(a_1a_2+c)\neq 0$.
\item $c \neq -\Tr_{q^n/q}(a_1a_2)\cdot n^{-1}$, where $n$ is seen as an element of $\F_q$.
\end{enumerate}
\end{proposition}
\begin{proof}
Recall that $a_1a_2$ is $q^n/q$-completely normal by Proposition~\ref{propo:hb2}. Next, fix some $c\in\F_q$. Clearly, if $c=-n\Tr_{q^n/q}(a_1a_2)$ (in the case $p\nmid n$), then, by Lemma~\ref{lemma:x+1}, $a_1a_2+c$ cannot be $q^n/q$-completely normal. So, from now on, we assume that either $p\mid n$, or $\Tr_{q^n/q}(a_1a_2+c)\neq 0$ if $p\nmid n$.
 
Now, take some $d\mid n$. If $p\mid (n/d)$, Lemma~\ref{lemma:x+1} yields that $a_1a_2+c$ is $q^n/q^d$-normal. So, we may focus on the case $p\nmid (n/d)$.

In this case, Lemma~\ref{lemma:x+1} implies that $a_1a_2+c$ is $q^n/q^d$-normal if $\Tr_{q^n/q^d}(a_1a_2+c)\neq 0$. However, if this is not the case, then
\[ \Tr_{q^n/q}(a_1a_2+c) = \Tr_{q^d/q}\left( \Tr_{q^n/q^d}(a_1a_2+c) \right) = \Tr_{q^d/q}(0) = 0, \]
a contradiction.
%
\end{proof}
%
%
\begin{remark}
The proof of Proposition~\ref{propo:a1a2-1} uses essentially the same arguments with that of Proposition~\ref{propo:hb2}, only with Lemma~\ref{lemma:x+1} in mind, instead of Lemma~\ref{lemma:hb2}.
\end{remark}
\begin{remark} \label{remark:non_primitive_theta}
For the purposes of this work, we are interested in finding $q^n/q$-completely normal elements, that have multiplicative order $q^n-1$. However, the element $a_1a_2$ that is obtained from Lemma~\ref{lemma:hb2} has order at most $(q^{n_1}-1)(q^{n_2}-1)/(q-1)$. In other words the construction of completely normal elements in the original proof of Theorem~\ref{thm:cnbt} never yields primitive elements. In our opinion this is a key factor for the persistence of Conjecture~\ref{conj:mm}.
\end{remark}
As a corollary of Proposition~\ref{propo:a1a2-1}, we obtain the main result of this section.
\begin{theorem} \label{thm:completely_normal_translates}
Let $\F_{q^n}/\F_q$ be a finite field extension and let $n=p_1^{n_1}\cdots p_k^{n_k}$ be the prime decomposition of $n$, with $k\geq 2$. Further, take some $q^{p_i^{n_i}}/q$-completely normal $a_i$, for every $i=1,\ldots ,k$ and set $a = a_1\cdots a_k$. 
\begin{enumerate}
\item If $p\mid n$, the set of translates $\mathcal T_a$ is comprised of $q^n/q$-completely normal elements.
\item If $p\nmid n$, then the set $\mathcal T_a\setminus\{ a -Tr_{q^n/q}(a)\cdot n^{-1}\}$ (where $n$ is seen as an element of $\F_q$), is comprised of $q^n/q$-completely normal elements
\end{enumerate}
\end{theorem}
\begin{proof}
Proposition~\ref{propo:hb2}, used inductively, ensures that $a_1\cdots a_{k-1}$ is $q^{p_1^{n_1}\cdots p_{k-1}^{n_{k-1}}}/q$-completely normal. Now, apply Proposition~\ref{propo:a1a2-1}.
\end{proof}
%
\begin{remark} \label{remark:tp}
Another way to attack Conjecture~\ref{conj:mm} emerges from Theorem~\ref{thm:completely_normal_translates}. Namely, Theorem~\ref{thm:completely_normal_translates} links the MM property with the translate property (and its strong version). The corresponding definitions and the details of this interesting connection are explained briefly in Appendix~\ref{app:tp}. However, as explained in more detail in Appendix~\ref{app:tp}, this connection is of small practical interest for the purposes of this work.
\end{remark}
  \section{Main condition} \label{sec:main}
Before we proceed further, we introduce the following family of finite field extensions that will be of interest for us.
\begin{definition}
Let $\F_{q^n}/\F_q$ be an extension of finite fields, such that
\begin{enumerate}
  \item $n=n_1n_2$, where $n_1,n_2>1$ with $\gcd(n_1,n_2)=1$ and
  \item $\F_{q^{n_2}}/\F_q$ is completely basic.
\end{enumerate} 
Then we call the extension \emph{$(n_1,n_2)$-partially completely basic}.
\end{definition} 
\begin{remark}
With the discussion on completely basic extensions in mind, see Sect.~\ref{sec:previous}, we observe that the family of partially completely basic extensions is rather large. In particular, if $n=r_1^{n_1}\cdots r_k^{n_k}$ is the prime factorization of $n$, if the extension $\F_{q^n}/\F_q$ is not completely basic, it is partially completely basic if $n_i=1$ or $2$ for some $i$, if $r_i=p$ for some $i$, or, if $r_i^{n_i}\mid q^n-1$ for some $i$, while it is clear a fixed extension may be $(n_1,n_2)$-partially completely basic for multiple parameters $n_1$ and $n_2$.
\end{remark}
%
%
%
We will also need the following technical definition.
\begin{definition}
Let $X$ be a positive integer or a polynomial in $\F_q[x]$. Then $W(X)$ denotes the number of squarefree divisors or the number of squarefree monic divisors of $X$, respectively.
\end{definition}
%
We are now in position to prove the main results of this section.
\begin{theorem} \label{thm:main_condition}
  Let $\F_{q^n}/\F_q$ be an $(n_1,n_2)$-partially completely basic extension, such that $p\nmid n$ and
%
  \[
  \frac{q-2}{q-1}\cdot q^{n_2/2} \geq 2 n_1 W(q')W(F_{n_2}') .
  \]
  Then $\F_{q^n}/\F_q$ possesses the MM property.
\end{theorem}
\begin{proof}
Theorem~\ref{thm:cnbt} ensures the existence of some $a\in\F_{q^{n_1}}$ that is $q^{n_1}/q$-completely normal. The fact that $\F_{q^{n_2}}/\F_q$ is completely basic and Proposition~\ref{propo:a1a2-1} imply that, if we identify some $b\in\F_{q^{n_2}}$ and $c\in\F_q^*$ such that
\begin{enumerate}
  \item $b$ is $q^{n_2}/q$-normal,
  \item $ab+c$ is primitive, and
  \item $\Tr_{q^n/q}(ab+c)\neq 0$,
\end{enumerate}
then $ab+c$ is primitive and $q^n/q$-completely normal. In particular, $\F_{q^n}/\F_q$ possesses the MM property.

%
From the characteristic functions of the corresponding properties, as presented in Sect.~\ref{sec:prelim}, it follows that the number of $b\in\F_{q^{n_2}}$ that combine the desired properties as above is
\[ \mathcal N = \sum_{b\in\F_{q^{n_2}}}\sum_{c\in\F_q^*} \omega (ab+c) \varOmega_{n_2}(b) z\left( \Tr_{q^n/q}(ab+c) \right) . \]
So, for our purposes, it suffices to show that $\mathcal N\neq 0$. We replace $\omega$ and $\varOmega_{n_2}$ with their expressions from Sect.~\ref{sec:prelim} and obtain:
\begin{align}
  \frac{\mathcal N}{\theta(q')} & = \sum_{d\mid q'} \frac{\mu(d)}{\phi(d)} \sum_{\substack{\chi\in\widehat{\F_{q^n}^*} \\ \ord(\chi) = d}} \sum_{b\in\F_{q^{n_2}}} \sum_{c\in\F_q^*} \chi(ab+c) \varOmega_{n_2}(b) z\left( \Tr_{q^n/q}(ab+c) \right) \nonumber \\ & = S_1 + S_2 , \label{eq:N=S1+S2}
\end{align}
where $S_1$ is the part that corresponds to $d=1$ and $S_2$ the one that corresponds to $d\neq 1$.

First we focus on $S_1$. We have that
\begin{equation} \label{eq:S1}
S_1 = \sum_{b\in\F_{q^{n_2}}} \varOmega_{n_2}(b) \sum_{c\in\F_q^*} z\left( \Tr_{q^n/q}(ab+c) \right)
 =  (q-2) \phi_q(F_{n_2}') . 
  \end{equation}

Then, we turn our attention to $S_2$. We have that
\[ S_2 = \frac{\theta_q(F_{n_2}')}{q} \sum_{\substack{d\mid q' \\ d\neq 1}} \sum_{f\mid F_{n_2}'} \frac{\mu(d)\mu_q(f)}{\phi(d)\phi_q(f)} \sum_{\substack{\chi\in\widehat{\F_{q^n}^*} \\ \ord(\chi) = d}} \sum_{\substack{\psi\in\widehat{\F_{q^{n_2}}} \\ \ord(\psi)=f}} \sum_{\eta\in\widehat{\F_q}} \sum_{c\in\F_q^*} \mathcal C(\chi,\psi,\eta,c) , \]
where
\[ \mathcal{C}(\chi,\psi,\eta,c) := \sum_{b\in\F_{q^{n_2}}} \chi(ab+c)\psi(b)\left( 1-\eta(\Tr_{q^n/q}(ab+c)) \right) . \]
It follows that
\begin{multline*}
 |\mathcal{C}(\chi,\psi,\eta,c)| \leq \\ |\chi(a)|\cdot \left| \sum_{b\in\F_{q^{n_2}}} \chi(a^{-1}c+b)\psi(b) \right| + \left| \chi(a)\eta\left( \Tr_{q^n/q}(c)\right)\right| \cdot \left| \sum_{b\in\F_{q^{n_2}}} \chi(a^{-1}c+b)\rho(b) \right| ,
 \end{multline*}
where $\rho(b) = \psi(b)\eta(\Tr_{q^n/n}(ab))$ is an additive character of $\F_{q^{n_2}}$.
Further, observe that, since $a$ is $q^{n_1}/q$-normal, Lemma~\ref{lemma:hb1} entails that it is also $q^n/q^{n_2}$-normal, thus $\F_{q^{n_2}}(a) = \F_{q^n}$, i.e., $\F_{q^{n_2}}(a^{-1}c) = \F_{q^n}$. It follows from Theorem~\ref{thm:char-sum} (applied on the extension $\F_{q^n}/\F_{q^{n_2}}$) that, since $\chi$ is not trivial, $|\mathcal{C}(\chi,\psi,\eta,c)|\leq 2 n_1 q^{n_2/2}$.
Furthermore, recall that there are exactly $\phi(d)$ (resp. $\phi_q(f)$) multiplicative (resp. additive) characters of order $d$ (resp. $q$-order $f$).
It follows that
\begin{align}
 |S_2| & \leq \frac{\theta_q(F_{n_2}')}{q} (W(q')-1)W(F_{n_2}') q (q-1) 2n_1 q^{n_2/2}  \nonumber \\
  & < \frac{\phi_q(F_{n_2}')}{q^{n_2/2}} W(q')W(F_{n_2}') (q-1) 2n_1 \label{eq:S2}
  \end{align}

We combine Eqs.~\eqref{eq:N=S1+S2}, \eqref{eq:S1} and \eqref{eq:S2} and obtain that $\mathcal N\neq 0$ if
\[ (q-2) \phi_q(F_{n_2}') \geq \frac{\phi_q(F_{n_2}')}{q^{n_2/2}} W(q')W(F_{n_2}') (q-1) 2 n_1 , \]
from where the result follows.
\end{proof}
For the case $p\mid n$ we have the following improved version of Theorem~\ref{thm:main_condition}.
\begin{theorem} \label{thm:main_condition_2}
  Let $\F_{q^n}/\F_q$ be an $(n_1,n_2)$-partially completely basic extension, such that $p\mid n$. Then, if
%
  \[
  q^{n_2/2} \geq n_1 W(q')W(F_{n_2}') ,
  \]
  $\F_{q^n}/\F_q$ possesses the MM property.
\end{theorem}
\begin{proof}
If $p\mid n$, the requirement $\Tr_{q^n/q}(ab+c)\neq 0$ is satisfied for all $c\in\F_q^*$. This means that the term $q-2$ that appears in Eq.~\eqref{eq:S1} can be replaced by $q-1$ and that the function $z$ can be skipped altogether. All the other arguments remain identical as the ones in the proof of Theorem~\ref{thm:main_condition}
\end{proof}
\begin{remark} \label{remark:q=2}
Note that Theorem~\ref{thm:main_condition} is ineffective in the case $q=2$ as its condition can never hold. In fact, our method fails in the case $q=2$ and $n$ odd. The underlying reason is that, under the assumptions of Theorem~\ref{thm:completely_normal_translates}, the set where we are looking for primitive completely normal elements is the singleton $\{ab\}$, but $ab$, as already mentioned in Remark~\ref{remark:non_primitive_theta}, is never primitive. The case $q=2$ and $n$ even can be, however, treated by Theorem~\ref{thm:main_condition_2}.
\end{remark}
\begin{remark} \label{remark:sieves}
In the literature one can find several powerful strategies to weaken the conditions of Theorems~\ref{thm:main_condition} and \ref{thm:main_condition_2}, such as the prime sieve \cite{cohenhuczynska03}, the modified prime sieve \cite{baileycohensutherlandtrudgian19} or the hybrid bound \cite{bagger24}. However, in this work, we choose to not use any of these in an attempt to keep the text as simple as possible focus on the novel aspects of this work.
\end{remark}
\section{An asymptotic result} \label{sec:asymptotic}
In this section, we will use Theorems~\ref{thm:main_condition} and \ref{thm:main_condition_2}, in order to obtain an asymptotic condition for an $(n_1,n_2)$-partially completely basic extension to possess the MM property. We will need the following classic result.
\begin{lemma}[\cite{apostol76}*{p.~296}] \label{lemma:apostol}
For every $\delta > 0$, $W(n) = o(n^\delta)$, where $o$ signifies the little-o notation.
\end{lemma}
We will also need the following lemmas.
\begin{lemma}[\cite{lenstraschoof87}*{Lemma~2.9}] \label{lemma:w(x^n-1)}
Let $q$ be a prime power and $n$ a positive integer. Then, we have $W(F_n') \leq 2^{\frac 12 (n+\gcd (n,q-1))}$.
In particular, $W (F_n') \leq 2^n$, while the equality holds if and only if $n \mid q - 1$.
Furthermore, if $n \nmid q - 1$, $W (F_n') \leq 2^{3n/4}$.
\end{lemma}
For small values of $q$, we will use the following.
\begin{lemma}[\cite{aguirreneumann21}*{Lemma~3.7}] \label{lemma:w(x^n-1):smallq}
Let $q$ be a prime power and $n$ a positive integer. Then, $W(F_n') \leq 2^{\frac{n+a}{b}}$ for some $a,b\in\Z$. In particular, for $q\geq 29$, we have $(a,b)=(0,1)$; for $7\leq q\leq 27$, we have $(a,b)=(q-1,2)$ and for $q\leq 5$ we have the following values for $a,b$:
\begin{center}
  \begin{tabular}{lll|lll}
    $q$ & $a$ & $b$  & $q$ & $a$ & $b$\\ \hline
    $2$ & $14$ & $5$ &
    $3$ & $20$ & $4$ \\
    $4$ & $12$ & $3$ &
    $5$ & $18$ & $3$
  \end{tabular}
\end{center}
\end{lemma}
We are now in position to prove the main result of this section.
\begin{theorem} \label{thm:main_asymptotic}
    Let $\F_{q^n}/\F_q$ be an $(n_1,n_2)$-partially completely basic extension, where $n_2$ is large compared to $n_1$, then $\F_{q^n}/\F_q$ possesses the MM property, unless $n$ is odd and $q=2$.
\end{theorem}
\begin{proof}
First we assume $q>2$. Fix some $n_1\geq 2$. Theorem~\ref{thm:main_condition}, along with Lemmas~\ref{lemma:apostol} and \ref{lemma:w(x^n-1)}, implies that an $(n_1,n_2)$-partially completely basic extension possesses the MM property, given that
\[ \left( \frac{\sqrt{q}}{2} \right)^{n_2} > 4 n_1 \cdot o(q^{n_2/4}) . \]
This implies the result for $q>4$. For $q=3$ and $4$, Lemma~\ref{lemma:w(x^n-1):smallq} yields $W(F_{n_2}') \leq 2^{\frac{n_2+20}{3}}$ and the above condition is replaced by  
\[ \left( \frac{\sqrt{q}}{\sqrt[3]{2}} \right)^{n_2} > 2^{26/3} n_1 \cdot o(q^{n_2/4}) , \]
which settles the cases $q=3$ and $4$. Finally, for $q=2$, we are confined to $n$ even, thus we may use Theorem~\ref{thm:main_condition_2} instead. Now, Lemma~\ref{lemma:w(x^n-1):smallq} entails $W(F_{n_2}') \leq 2^{\frac{n_2+14}{5}}$ and the above condition is replaced by  
\[ 2^{3n_2/10} > 2^{14/5} n_1 \cdot o(2^{n_2/5}) , \]
which completes our proof.
\end{proof}
\section{Computational aspects} \label{sec:comp}
In this section we explore the computational aspects of Theorem~\ref{thm:main_condition}.
First, with the condition of Theorem~\ref{thm:main_condition} in mind, we see that a concrete expression of  Lemma~\ref{lemma:apostol} is needed.
	
	\begin{lemma}[\cite{cohenhuczynska03}*{Lemma~3.7}] \label{lemma:W(n)}
	   For any $\alpha \in \mathbb{N}$ and a positive real number $\nu$, $W(\alpha) \leq \C_\nu^{(\alpha)}\cdot \alpha^{1/\nu}$, where $\C_\nu^{(\alpha)} = \prod_{i=1}^{t} 2/(p_i^{1/\nu})$ and $p_1, p_2, \dots, p_t$ are the primes less than or equal to $2^\nu$ that divide $\alpha$.
	   \end{lemma}
	  Let $p_1,\ldots ,p_s$ be \emph{all} the primes less or equal to $2^\nu$. We note that, in this section, in addition to the number $\mathcal C_\nu^{(\alpha)}$ defined above we will use the following two numbers:
	  \begin{enumerate}
	  \item $\mathcal C_\nu := \prod_{i=1}^{s} 2/(p_i^{1/\nu})$. Clearly, $\mathcal C_\nu^{(\alpha)} \leq \mathcal C_\nu$ for all $\alpha\in\Z$.
	  \item For any $1\leq j\leq s$, $\mathcal C_\nu^{[p_j]} := \mathcal C_{\nu} p_j^{1/\nu}/2$. Clearly, $\mathcal C_\nu^{(\alpha)} \leq \mathcal C_\nu^{[p_j]}$ for all $\alpha\in\Z$ with $p_j\nmid\alpha$.
	  \end{enumerate}
	  Furthermore, the number $\mathcal C_\nu^{(\alpha)}$ in the statement of Lemma~\ref{lemma:W(n)} can be freely replaced by either $\mathcal C_\nu$, or, if the prime $r\leq 2^\nu$ does not divide $\alpha$, by $\mathcal C_\nu^{[r]}$.

\subsection{An algorithm for the MM property} \label{subsec:algo}

Now, we present an algorithm, that takes $n_1$ as its input and aims to establish that all $(n_1,n_2)$-partially completely basic extensions possess the MM property is given.
The algorithm is comprised by the  following steps:
\begin{enumerate}[label=\textbf{Step \arabic*}, ref=\arabic*]
	\item Set $n_{\min}$ as the minimum value of $n_2$ that is not covered by Theorem~\ref{thm:hh} for the given $n_1$.
	\item For $n_2=n_{\min}$, find $q_{\max}$, the minimum prime power $q$ that satisfies the condition of Theorem~\ref{thm:main_condition}, where the numbers $W(F_{n_2}')$ and $W(q')$ are estimated by Lemmas~\ref{lemma:w(x^n-1)} and \ref{lemma:W(n)}, respectively. We note that, with the notation of Lemma~\ref{lemma:W(n)}, we take $\nu=4n_1$ and we bound the constant $\mathcal C_{4n_1}^{(q^{n_1n_2}-1)}$ by the generic bound $\mathcal C_{4n_1}$. At this point we are left with a finite number of prime powers $q$ that are not covered.
	\item For each prime power $3\leq q< q_{\max}$, find the maximum value of $n_2$, $n_{\max}^{(q)}$, not satisfying the condition of the previous step. Note that in this step, every $q$ is treated as fixed, the sharper estimate $\mathcal C_{4n_1}^{[p]}$ (where $p$ stands for the unique prime divisor of $q$) is employed over the generic $\mathcal{C}_{4n_1}$ and $n_2$ is treated as a variable. Upon completion, for each prime power $3\leq q< q_{\max}$, we are left with a range $n_{\min}\leq n_2\leq n_{\max}^{(q)}$ of yet unsettled cases.\footnote{Here, observe that since the sharper $\mathcal C_{4n_1}^{[p]}$ is favored over the generic $\mathcal{C}_{4n_1}$ of the previous step, it is possible that $n_{\max}^{(q)} < n_{\min}$; this means that the corresponding prime power $q$ is settled.}\label{s3}
	\item For each prime power $3\leq q\leq q_{\max}$ and integers $n_{\min}\leq n_2\leq n_{\max}^{(q)}$, we explicitly identify which values of $n_2$ are such that \label{s4}
	\begin{itemize}
		\item $\F_{q^{n_2}}/\F_q$ is completely basic,
		\item $\gcd(n_1,n_2)=1$, and
		\item $q^{n_1n_2}>10^{80}$.
	\end{itemize}
	The first two conditions ensure that the extension $\F_{q^{n_1n_2}}\F_q$ is $(n_1,n_2)$-partially completely basic and the third one that the extension is not already covered by Theorem~\ref{thm:hh}. At this point, we are left with a finite list of pairs $(q,n_2)$ that correspond to the $(n_1,n_2)$-partially completely basic extensions not already covered.
	\item For each of these pairs we check the condition of Theorem~\ref{thm:main_condition}, where $W(F_{n_2}')$ is explicitly computed and $W(q')$ is estimated by Lemma~\ref{lemma:W(n)}, where, under the notation Lemma~\ref{lemma:W(n)}, $\nu=4n_1$ and the constant $\mathcal C_{4n_1}^{(q^{n_1n_2}-1)}$ is explicitly computed. Remove from the list of possible exceptions the pairs that pass this test. \label{s5}
	\item For the remaining pairs we check the condition of Theorem~\ref{thm:main_condition}, where $W(F_{n_2}')$ and $W(q')$ are both explicitly computed. Remove from the list of possible exceptions the pairs that pass this test. \label{s6}
	\item If the list of possible exceptions remain nonempty at this point, return \textbf{FAIL}.
	\item Now, we move on to the case $q=2$. We repeat Steps~\ref{s3}--\ref{s6} with the difference that Theorem~\ref{thm:main_condition_2} is used over Theorem~\ref{thm:main_condition} and that, in Step~\ref{s4}, we additionally demand $n_2$ to be even. \label{s8}
	\item If the list of possible exceptions remain nonempty at this point, return \textbf{FAIL}. Otherwise, return \textbf{SUCCESS}.
\end{enumerate}
The rest of this section is dedicated into applying this algorithm in an attempt to establish the MM property for $(n_1,n_2)$-partially completely basic extensions for small values of $n_1$.
\subsection{Concrete results} \label{subsec:concrete}
We successfully applied the algorithm of Subsect.~\ref{subsec:algo} for $n_1=2$ and $n_1=3$, while we obtained partial results for $n_1=4$. Namely, we establish the following theorem.
\begin{theorem} \label{thm:main_concrete}
    Let $\F_{q^n}/\F_q$ be an $(n_1,n_2)$-partially completely basic extension. Then $\F_{q^n}/\F_q$ possesses the MM property in the following cases:
    \begin{enumerate}
    \item $n_1=2$ or $3$.
    \item $n_1=4$ and either
    	\begin{itemize}
		\item $q\geq 199211272511189639\,(\approx 2\cdot 10^{17})$,
		\item $n_2\geq n_1^{n_1} = 256$ and $q\geq 22271$, or,
		\item $n_2\geq n_1^6 = 4096$.
		\end{itemize} 
    \end{enumerate}
\end{theorem}
For our calculations we used the \textsc{SageMath} system, running on a modern laptop. In the following subsections we present the most important details of our calculations.
\subsubsection{The cases $n_1=2$ and $n_1=3$}
 The most essential information of our implementation for the cases $n_1=2$ and $n_1=3$ is summarized in Table~\ref{tab:summary}.
\begin{table}[h!]
\begin{center}
\begin{tabular}{l|lll|p{1.8em}p{1.8em}p{1.8em}p{1.8em}p{1.8em}p{1.8em}|l}
 \multicolumn{4}{c}{} & \multicolumn{6}{c}{\#{} of pairs after Step} & {} \\ 
  $n_1$ & $n_{\min}$ & $\mathcal C_{4n_1}$ & $q_{\max}$ & \ref{s4} & \ref{s5} & \ref{s6}& \ref{s8}.\ref{s4} & \ref{s8}.\ref{s5} & \ref{s8}.\ref{s6} & time \\ \hline
  $2$ & $102$ & $4514.6265$ & $11$ & 15 & 0 & 0 & 0 & 0 & 0 & $\sim$1~sec. \\ 
  $3$ & $67$ & $1.057\cdot 10^{24}$ & $487$ & 3011 & 0 & 0 & 111 & 0 & 0 & $\sim$30~sec.
\end{tabular}
\end{center}
  \caption{Summary of the implementation of the algorithm of Subsect.~\ref{subsec:algo} for $n_1=2$ and $n_1=3$.\label{tab:summary}}
\end{table}

We note that the zeros appearing in the columns under Steps~\ref{s6}, and \ref{s8}.\ref{s6} show that the application was successful. The zeros in the columns under Steps~\ref{s5}, and \ref{s8}.\ref{s5} show that we never had to resort to computing $W(q')$, which is an expensive computation, as the computer should first factor $q^{n_1n_2}-1$ into primes. Furthermore, the zero in the entry that corresponds to $n_1=2$ and Step~\ref{s8}.\ref{s4} is the result of the fact that, in this case, we expect $n_2$ to be both odd (as it has to be relatively prime to $n_1=2$) and even (as it has to be divided by $p=2$).

Next, in Table~\ref{tab:3}, we present the numbers appearing in Step~\ref{s3} for $n_1=2$ and, in Table~\ref{tab:2}, we present the numbers appearing in Step~\ref{s3} for $n_1=3$. We note that in Table~\ref{tab:3}, the case $q=2$ is missing, as, we already explained, this case is irrelevant and, in Table~\ref{tab:2}, the entries that correspond to $q=2$ derive from Step~\ref{s8}.\ref{s3}.
\begin{table}[h!]
  \begin{center}
      \begin{tabular}{lll|lll|lll}
        $q$ & $n_{\max}^{(q)}$ & $\C_{4n_1}^{[p]}$ & $q$ & $n_{\max}^{(q)}$ & $\C_{4n_1}^{[p]}$ & $q$ & $n_{\max}^{(q)}$ & $\C_{4n_1}^{[p]}$ \\ \hline
 $3$ & $133$ & $2589.5959$ &
 $4$ & $108$ & $2461.6176$ &
 $5$ & $81$ & $2760.3433$ \\
 $7$ & $84$ & $2878.9167$ &
 $8$ & $68$ & $2461.6176$ &
 $9$ & $60$ & $2589.5959$ 
      \end{tabular}
  \end{center}
  \caption{Prime powers $3 \leq q<q_{\max}$, the corresponding $n_{\max}^{(q)}$ and the corresponding value of $\C_{4n_1}^{[p]}$ for $n_1=2$.\label{tab:3}}
\end{table}

  \begin{center}
\begin{longtable}{lll|lll|lll}
        $q$ & $n_{\max}^{(q)}$ & $\C_{4n_1}^{[p]}$ & $q$ & $n_{\max}^{(q)}$ & $\C_{4n_1}^{[p]}$ & $q$ & $n_{\max}^{(q)}$ & $\C_{4n_1}^{[p]}$ \endfirsthead
        \multicolumn{9}{l}{Continued from previous page.} \\ \\
        $q$ & $n_{\max}^{(q)}$ & $\C_{4n_1}^{[p]}$ & $q$ & $n_{\max}^{(q)}$ & $\C_{4n_1}^{[p]}$ & $q$ & $n_{\max}^{(q)}$ & $\C_{4n_1}^{[p]}$ \\ \hline
        \endhead
        \\ \multicolumn{9}{l}{Continued on next page.} \endfoot
        \endlastfoot
         \hline
  $2$ & $1666$ & $5.6008\cdot 10^{23}$ &
 $3$ & $599$ & $5.7933\cdot 10^{23}$ &
 $4$ & $517$ & $5.6008\cdot 10^{23}$ \\
 $5$ & $357$ & $6.0453\cdot 10^{23}$ &
 $7$ & $421$ & $6.2172\cdot 10^{23}$ &
 $8$ & $341$ & $5.6008\cdot 10^{23}$ \\
 $9$ & $294$ & $5.7933\cdot 10^{23}$ &
 $11$ & $238$ & $6.4558\cdot 10^{23}$ &
 $13$ & $207$ & $6.5463\cdot 10^{23}$ \\
 $16$ & $179$ & $5.6008\cdot 10^{23}$ &
 $17$ & $173$ & $6.6943\cdot 10^{23}$ &
 $19$ & $162$ & $6.7566\cdot 10^{23}$ \\
 $23$ & $148$ & $6.8651\cdot 10^{23}$ &
 $25$ & $142$ & $6.0453\cdot 10^{23}$ &
 $27$ & $138$ & $5.7933\cdot 10^{23}$ \\
 $29$ & $382$ & $6.9990\cdot 10^{23}$ &
 $31$ & $344$ & $7.0380\cdot 10^{23}$ &
 $32$ & $327$ & $5.6008\cdot 10^{23}$ \\
 $37$ & $271$ & $7.1425\cdot 10^{23}$ &
 $41$ & $242$ & $7.2039\cdot 10^{23}$ &
 $43$ & $230$ & $7.2325\cdot 10^{23}$ \\
 $47$ & $211$ & $7.2863\cdot 10^{23}$ &
 $49$ & $203$ & $6.2172\cdot 10^{23}$ &
 $53$ & $190$ & $7.3597\cdot 10^{23}$ \\
 $59$ & $175$ & $7.4257\cdot 10^{23}$ &
 $61$ & $170$ & $7.4464\cdot 10^{23}$ &
 $64$ & $163$ & $5.6008\cdot 10^{23}$ \\
 $67$ & $159$ & $7.5048\cdot 10^{23}$ &
 $71$ & $153$ & $7.5412\cdot 10^{23}$ &
 $73$ & $150$ & $7.5587\cdot 10^{23}$ \\
 $79$ & $143$ & $7.6086\cdot 10^{23}$ &
 $81$ & $140$ & $5.7933\cdot 10^{23}$ &
 $83$ & $139$ & $7.6400\cdot 10^{23}$ \\
 $89$ & $133$ & $7.6845\cdot 10^{23}$ &
 $97$ & $127$ & $7.7399\cdot 10^{23}$ &
 $101$ & $124$ & $7.7660\cdot 10^{23}$ \\
 $103$ & $123$ & $7.7787\cdot 10^{23}$ &
 $107$ & $120$ & $7.8034\cdot 10^{23}$ &
 $109$ & $119$ & $7.8155\cdot 10^{23}$ \\
 $113$ & $117$ & $7.8390\cdot 10^{23}$ &
 $121$ & $112$ & $6.4558\cdot 10^{23}$ &
 $125$ & $111$ & $6.0453\cdot 10^{23}$ \\
 $127$ & $110$ & $7.9156\cdot 10^{23}$ &
 $128$ & $109$ & $5.6008\cdot 10^{23}$ &
 $131$ & $109$ & $7.9361\cdot 10^{23}$ \\
 $137$ & $106$ & $7.9658\cdot 10^{23}$ &
 $139$ & $106$ & $7.9754\cdot 10^{23}$ &
 $149$ & $102$ & $8.0217\cdot 10^{23}$ \\
 $151$ & $102$ & $8.0306\cdot 10^{23}$ &
 $157$ & $100$ & $8.0568\cdot 10^{23}$ &
 $163$ & $98$ & $8.0820\cdot 10^{23}$ \\
 $167$ & $97$ & $8.0983\cdot 10^{23}$ &
 $169$ & $97$ & $6.5463\cdot 10^{23}$ &
 $173$ & $96$ & $8.1222\cdot 10^{23}$ \\
 $179$ & $95$ & $8.1453\cdot 10^{23}$ &
 $181$ & $94$ & $8.1528\cdot 10^{23}$ &
 $191$ & $92$ & $8.1894\cdot 10^{23}$ \\
 $193$ & $92$ & $8.1966\cdot 10^{23}$ &
 $197$ & $91$ & $8.2106\cdot 10^{23}$ &
 $199$ & $91$ & $8.2175\cdot 10^{23}$ \\
 $211$ & $89$ & $8.2577\cdot 10^{23}$ &
 $223$ & $87$ & $8.2958\cdot 10^{23}$ &
 $227$ & $86$ & $8.3081\cdot 10^{23}$ \\
 $229$ & $86$ & $8.3142\cdot 10^{23}$ &
 $233$ & $85$ & $8.3262\cdot 10^{23}$ &
 $239$ & $85$ & $8.3439\cdot 10^{23}$ \\
 $241$ & $84$ & $8.3497\cdot 10^{23}$ &
 $243$ & $84$ & $5.7933\cdot 10^{23}$ &
 $251$ & $83$ & $8.3780\cdot 10^{23}$ \\
 $256$ & $82$ & $5.6008\cdot 10^{23}$ &
 $257$ & $82$ & $8.3945\cdot 10^{23}$ &
 $263$ & $82$ & $8.4107\cdot 10^{23}$ \\
 $269$ & $81$ & $8.4265\cdot 10^{23}$ &
 $271$ & $81$ & $8.4317\cdot 10^{23}$ &
 $277$ & $80$ & $8.4471\cdot 10^{23}$ \\
 $281$ & $80$ & $8.4572\cdot 10^{23}$ &
 $283$ & $80$ & $8.4622\cdot 10^{23}$ &
 $289$ & $79$ & $6.6943\cdot 10^{23}$ \\
 $293$ & $79$ & $8.4867\cdot 10^{23}$ &
 $307$ & $78$ & $8.5198\cdot 10^{23}$ &
 $311$ & $77$ & $8.5290\cdot 10^{23}$ \\
 $313$ & $77$ & $8.5336\cdot 10^{23}$ &
 $317$ & $77$ & $8.5426\cdot 10^{23}$ &
 $331$ & $76$ & $8.5734\cdot 10^{23}$ \\
 $337$ & $75$ & $8.5863\cdot 10^{23}$ &
 $343$ & $74$ & $6.2172\cdot 10^{23}$ &
 $347$ & $74$ & $8.6072\cdot 10^{23}$ \\
 $349$ & $74$ & $8.6113\cdot 10^{23}$ &
 $353$ & $74$ & $8.6195\cdot 10^{23}$ &
 $359$ & $74$ & $8.6316\cdot 10^{23}$ \\
 $361$ & $73$ & $6.7566\cdot 10^{23}$ &
 $367$ & $73$ & $8.6475\cdot 10^{23}$ &
 $373$ & $73$ & $8.6592\cdot 10^{23}$ \\
 $379$ & $72$ & $8.6707\cdot 10^{23}$ &
 $383$ & $72$ & $8.6783\cdot 10^{23}$ &
 $389$ & $72$ & $8.6896\cdot 10^{23}$ \\
 $397$ & $71$ & $8.7043\cdot 10^{23}$ &
 $401$ & $71$ & $8.7116\cdot 10^{23}$ &
 $409$ & $71$ & $8.7259\cdot 10^{23}$ \\
 $419$ & $70$ & $8.7435\cdot 10^{23}$ &
 $421$ & $70$ & $8.7470\cdot 10^{23}$ &
 $431$ & $70$ & $8.7641\cdot 10^{23}$ \\
 $433$ & $70$ & $8.7675\cdot 10^{23}$ &
 $439$ & $69$ & $8.7776\cdot 10^{23}$ &
 $443$ & $69$ & $8.7842\cdot 10^{23}$ \\
 $449$ & $69$ & $8.7941\cdot 10^{23}$ &
 $457$ & $68$ & $8.8070\cdot 10^{23}$ &
 $461$ & $68$ & $8.8134\cdot 10^{23}$ \\
 $463$ & $68$ & $8.8166\cdot 10^{23}$ &
 $467$ & $68$ & $8.8229\cdot 10^{23}$ &
 $479$ & $67$ & $8.8416\cdot 10^{23}$ \\
  \caption{Prime powers $2 \leq q<q_{\max}$, the corresponding $n_{\max}^{(q)}$ and the corresponding value of $\C_{4n_1}^{[p]}$ for $n_1=3$.}\label{tab:2}
\end{longtable}
 \end{center}
\subsubsection{The case $n_1=4$}
Finally, we comment on our attempt in implementing the algorithm of Subsect.~\ref{subsec:algo} on the $n_1=4$ case. The main information of our attempt on the $n_1=4$ case is summarized in Table~\ref{tab:summary2}.

First, we attempted to run algorithm of Subsect.~\ref{subsec:algo} as it is. So, we computed $n_{\min} = 50$, $\mathcal C_{4n_1} = 1.9356\cdot 10^{200}$ and $q_{\max} = 199211272511189639 \approx 2\cdot 10^{17}$. However, the vast size of $q_{\max}$ makes it clear that even storing the numbers $n_{\max}^{(q)}$ for each $3\leq q\leq q_{\max}$ is not straightforward and, in fact, our computer crashed during the computation of these numbers.

Then, we confined ourselves (arbitrarily) to $n_{\min} = n_1^{n_1} = 256$. In this case, we computed $q_{\max} = 22271$. This made the completion of Step~\ref{s3} feasible with a total number of 82148 possible exceptional pairs. However, due to the large size of these numbers, our computer failed to complete Step~\ref{s4} after having examined about 2200 pairs, with no exceptional pair found until this point.

Our attempt on the $n_{\min} = n_1^5 = 1024$ case faced similar difficulties, although, in this case, we got a significantly smaller $q_{\max} = 27$ and after Step~\ref{s3} we had 5936 possible exceptional pairs. This time, only around 1650 of them were in fact checked during Step~\ref{s4}, before our computer's crash, without any exceptions up to this point.

Finally, we successfully run the algorithm with the twist that we chose $n_{\min}=n_1^6 = 4096$, where we obtained $q_{\max} = 4$. Also, we note that in this case, we got $n_{\max}^{(2)} = 13385$, $\mathcal C_{4n_1}^{(2)} = 1.0106\cdot 10^{200}$, $n_{\max}^{(3)} = 4605$, and $\mathcal C_{4n_1}^{(3)} = 1.0366\cdot 10^{200}$, respectively.
This completes the proof of Theorem~\ref{thm:main_concrete}. 

\begin{table}[h!]
\begin{center}
\begin{tabular}{ll|p{2.5em}p{2.5em}p{2.5em}p{2.5em}p{2.5em}p{2.5em}|l}
 \multicolumn{2}{c}{} & \multicolumn{6}{c}{\#{} of pairs after Step} & {} \\ 
   $n_{\min}$ & $q_{\max}$ & \ref{s4} & \ref{s5} & \ref{s6}& \ref{s8}.\ref{s4} & \ref{s8}.\ref{s5} & \ref{s8}.\ref{s6} & time \\ \hline
   $50$ & $\approx 2\cdot 10^{17}$ & - & - & - & - & - & - & - \\ 
   $256$ & $22271$ & 82148 & - & - & - & - & - & - \\
   $1024$ & $27$ & 5936 & - & - & - & - & - & - \\
   $4096$ & $4$ & 188 & 0 & 0 & 3317 & 0 & 0 & $\sim$8 min.
\end{tabular}
\end{center}
  \caption{Summary of the implementation of the algorithm of Subsect.~\ref{subsec:algo} for $n_1=4$, with various choices of $n_{\min}$.\label{tab:summary2}}
\end{table}
\section{Discussion} \label{sec:discussion}
We conclude this work with a short discussion about possible future research directions regarding Conjecture~\ref{conj:mm}.

The first one is hinted in Remark~\ref{remark:sieves}. In particular, we believe that, with considerably additional computational efforts and resources, Theorem~\ref{thm:main_concrete} can potentially be extended to $n_1=4$. However, increasing $n_1$ further should be feasible only after implementing the methods described in Remark~\ref{remark:sieves}. Our impression is that these methods should be able to increase $n_1$ even to two-digit numbers.


Also, we stress that in the literature, for example see \cite{aravena20}, one can find additional constructions of completely normal elements. Possibly, these constructions can be exploited into establishing the MM property for additional finite fields extensions, or even into completely resolving Conjecture~\ref{conj:mm}.
\section*{Acknowledgment}
We are grateful to D.~Hachenberger for pointing out a serious mistake in our original manuscript and his useful comments.
  \bibliography{refs_mm.bib} 
\appendix
\section{The translate properties and their connection with the MM property} \label{app:tp}
In this appendix, we dive deeper in the perspective raised in Remark~\ref{remark:tp}.
We start with defining the properties that are of interest for us.
\begin{definition}
Let $\F_{q^n}/\F_q$ be a finite field extension such that, for every $a\in\F_{q^n}$ such that $\F_q(a) = \F_{q^n}$, some primitive element lies within the set of translates $\mathcal T_a$. Then the extension possesses the \emph{translate property}. If, every such set of translates $\mathcal T_a$ contains two distinct primitive elements, then the extension possesses the \emph{strong translate property}.
\end{definition}
Furthermore, it is known that, for fixed $n$, if $q$ is large enough, then the extension $\F_{q^n}/\F_q$ possesses the translate property.
\begin{theorem}[The Translate Property] \label{thm:tp}
Let $n$ be a positive integer. There exists some $\mathrm{TP}(n)$ such that for every prime power $q>\mathrm{TP}(n)$, the extension $\F_{q^n}/\F_q$ possesses the translate property.
\end{theorem}
The above was initially established by Davenport~\cite{davenport37} for prime $q$ and extended to its stated form by Carlitz~\cite{carlitz53a}. In a similar fashion, in this work, we prove the strong version as follows.
\begin{theorem}[The Strong Translate Property] \label{thm:stp}
Let $n$ be a positive integer. There exists some $\mathrm{STP}(n) \geq \mathrm{TP}(n)$ such that for every prime power $q>\mathrm{STP}(n)$, the extension $\F_{q^n}/\F_q$ possesses the strong translate property.
\end{theorem}
\begin{proof}
We assume that $q>\mathrm{TP}(n)$ and fix some $b\in\F_{q^n}$, such that $\F_q(b) = \F_{q^n}$. We will show that, if $q$ is large enough, $\mathcal T_b$ contains two primitive elements.
Theorem~\ref{thm:tp} implies that, for some $c\in\F_q$, the element $a=b+c$ is primitive, while clearly $\mathcal T_b = \mathcal T_a$, that is, if the set $\mathcal T_a\setminus\{ a\}$ contains a primitive element, our proof is complete.

The number of primitive elements within the set $\mathcal T_a \setminus\{ a\}$ is
\begin{align} 
\mathcal N_{\mathrm{STP}} & = \sum_{c\in\F_q^*} \omega(a +c) \nonumber \\
& = \theta(q') \sum_{d\mid q'}\frac{\mu(d)}{\phi(d)} \sum_{\substack{\chi\in\widehat{\F_{q^n}^*}\\ \ord(\chi)=d}} \sum_{c\in\F_q^*} \chi(a+c) \nonumber \\
& = \theta(q') (S_1 + S_2) , \label{eq:app1}
\end{align}
where $q'=q^n-1$, $S_1$ is the part of the sum that corresponds to $d=1$ and $S_2$ is the remaining sum.

Clearly, $S_1=q-1$. Further, the fact that, for $d\mid q'$, there are $\phi(d)$ multiplicative characters of order $d$ and Theorem~\ref{thm:char-sum} yield $|S_2|\leq W(q')(nq^{1/2}+1)$, hence, $|S_2|<2nW(q')\sqrt{q}$.

Now, Lemma~\ref{lemma:apostol} yields that, for $q$ large enough, $S_1 = q-1 > |S_2|$, which, combined with Eq.~\eqref{eq:app1} implies $\mathcal N_{\mathrm{STP}}\neq 0$. The result follows.
\end{proof}
Now, observe that Theorem~\ref{thm:completely_normal_translates} yields that, if $p\mid n$, the translate property  implies the MM property,
while, in any case, the strong translate property implies the MM property.
In other words, we have the following.
\begin{theorem} \label{thm:stpmm}
Let $\F_{q^n}/\F_q$ be a finite field extension, such that, either
\begin{enumerate}
\item $p\mid n$ and $q>\mathrm{TP}(n)$, or,
\item $q>\mathrm{STP}(n)$,
\end{enumerate}
then this extension possesses the MM property.
\end{theorem}
\begin{remark}
Despite the fact that Theorem~\ref{thm:stpmm} might appear like a meaningful way to attack Conjecture~\ref{conj:mm}, this is, unfortunately, not the case. This is due to the fact that, although little is known about the numbers $\mathrm{TP}(n)$ and only a handful of them are known, these numbers suggest that, in principle, they tend to be significantly larger than $n$. In particular, see \cite{baileycohensutherlandtrudgian19} and the references therein, we know that $\mathrm{TP}(2)=1$, $\mathrm{TP}(3)=37$, and $43\leq \mathrm{TP}(4)\leq 102829$. In other words, Theorems~\ref{thm:hb1}, \ref{thm:gk1}, and \ref{thm:gk2} probably already cover the extensions that are known to possess the (strong) translate property.
\end{remark}
\end{document}